\documentclass[10pt,oneside,a4paper]{article}
\usepackage{amsmath}
\usepackage{enumerate}
\usepackage{multicol}
\usepackage{parallel}
\usepackage{fancyhdr}
\usepackage{graphicx}
\usepackage{hyperref}
\usepackage{array}
\usepackage{rotating}
\usepackage{amsfonts,amssymb,amsthm,epsfig,url}

\theoremstyle{plain}
\newtheorem{thm}{Theorem}[section] % reset theorem numbering for each chapter
\usepackage{xcolor}
\theoremstyle{definition}
\newtheorem{lemma}[thm]{Lemma}

\begin{document}
\begin{center}
 \Large \bf A Mathematical Model for the Co-infection Dynamics of \emph{Pneumocystis} Pneumonia  and HIV/AIDS with Treatment
 \normalsize\normalfont\\{\color{white} hhhhh}\\
  { Michael Byamukama \footnote{*Corresponding author: mbyamukama@must.ac.ug}, Damian Kajunguri$^2$, Martin Karuhanga$^1$\\$^1$Department of Mathematics, Mbarara University of Science and Technology,}\\{P.O.Box 1410, Mbarara, Uganda.}\\
$^2$Department of Mathematics,  Kabale University, P.O. Box 317, Kabale, Uganda.\\
\textbf{Emails:} mbyamukama@must.ac.ug, dkajunguri@kab.ac.ug, mkaruhanga@must.ac.ug\\
ORCID1: 0000-0002-1357-4797, ORCID2: 0000-0001-7205-2561, ORCID3: 0000-0002-7254-9073
\end{center}

\begin{abstract}The control of opportunistic infections among HIV infected individuals should be one of the major public health concerns in reducing  mortality rate of individuals living with HIV/AIDS. In this study a deterministic co-infection mathematical model is developed to provide a quantification of treatment at each contagious stage against  \emph{Pneumocystis} Pneumonia  (PCP) among HIV infected individuals on ART. The goal is to minimize the co-infection burden by putting the curable PCP under control. The disease-free equilibria for the HIV/AIDS sub-model, PCP sub-model and the co-infection model are shown to be locally asymptotically stable when their associated disease threshold parameter is less than a unity. By use of suitable Lyapunov functions, the endemic equilibria corresponding to  HIV/AIDS and PCP sub-models are globally asymptotically stable whenever the HIV/AIDS related basic reproduction number $\mathcal{R}_{0H}$ and  the PCP related reproduction number $\mathcal{R}_{0P}$ are respectively greater than a unity.  The sensitivity analysis results implicate that the effective contact rates are the main mechanisms fueling the proliferation of the two diseases and on the other hand treatment efforts play an important role in reducing the incidence. The model numerical results reveal that PCP carriers have a considerable contribution in the transmission dynamics of PCP.  Furthermore,  treatment of PCP at all contagious phases significantly reduces the burden with HIV/AIDS and PCP co-infection. \\
\textbf{Key words:} HIV/AIDS, PCP Carriers, Basic reproduction number(s), Co-infection, Stability, Sensitivity analysis.\\
\textbf{AMS subject classification 2010:}34D23, 92D30, 93A030
\end{abstract}

\section{Introduction}\label{sec1}
In absence of a curative formula, the Human Immunodeficiency Virus (HIV) that leads to Acquired Immunodeficiency Syndrome (AIDS) has traumatized the Sub-Saharan Africa since its discovery in 1981.  HIV/AIDS persists to be a prime global public health concern, having claimed more than tens of millions lives so far and 650,000 people in 2021 alone \cite{i2}. There were 38.4 million people living with HIV/AIDS in 2021 of which the Sub-Saharan Africa accounts for almost 70\% of the global HIV/AIDS infections \cite{i2}.\\The discovery of HIV as a pathogen that leads to AIDS and subsequent general utilization of anti-retroviral therapy (ART) have altered HIV from a definite cause of death  to a manageable infection and this has significantly improved the longevity of HIV infected individuals \cite{b20,b19}.
While HIV does not kill, if left untreated it minimizes the CD4 cell count in the body, which renders  immune system incapable of  fighting off infections and hence  infected individuals advance to the fatal AIDS stage. This gives rise to many opportunistic infections such as \emph{Pneumocystis} pneumonia (PCP), candidiasis, tuberculosis, cytomegalovirus, Hepatitis, and cancers like lymphoma, kaposis sarcoma, the list is infinite.
Thus,  early medication of  HIV infectives with antiretroviral therapy (ART) can lower the viral load set point and thus prolonging the life of the infectives \cite{i3, b19}.\\
In the last two decades, considerable global efforts have been mounted to address  the HIV/AIDS epidemic, and significant progress has been made. The number of people newly infected with HIV, especially children, and the number of AIDS-related deaths have declined over the years, and the number of people with HIV receiving treatment increased to 28.7 million in 2021 \cite{i2}.

\emph{Pneumocystis} pneumonia (PCP)  is a  potentially life-threatening pulmonary infection  caused by the fungus \emph{Pneumocystis jirovecii} and has long been recognized in immunocompromised individuals and HIV-infected patients with a low CD4 cell count \cite{i1,i5,i4}. \emph{Pneumocystis jiroveci} develops via airborne transmission or reactivation of improperly treated infection \cite{b21}. The clinical signs and symptoms include fever, shortness of breath with \emph{hypoxemia}, and non-productive or dry cough. Specific diagnosis of PCP is possible using respiratory specimens with direct immunofluorescent staining and invasive procedures are required in order to prevent unnecessary treatments \cite{b21,b22}.
There is no vaccine to fend off PCP but  once infected, trimethoprim and sulfamethoxazole is the first-line agent for treatment \cite{i1}. To individuals living with HIV, trimethoprim sulfamethoxazole can be administered daily to reduce the risk of contracting PCP \cite{i5, i2}.

Co-infections with HIV/AIDS and other linked diseases have evoked attention since the concurrent infection of one host with more than one different pathogens present shattering effects to the affected individual. The effect of ART and other preventive means on HIV/AIDS patients have been studied widely in \cite{b13,b10,b14,b4} and the results are promising so long as individuals do not default the treatment.\\
Not many researchers have studied the co-infection mathematical models for the  dynamics of HIV/AIDS and pneumonia \cite{b12, b1, b3},  other available works are clinical studies \cite{b27,b29,b31,b30} and systematic reviews \cite{b28,b32,b33}.
In the study by Ntiiri \emph{et al}. \cite{b12}, the maximum protection against pneumonia and maximum protection against HIV/ AIDS would lower the rate of disease prevalence. Teklu and Mekonnen \cite{b1} developed a deterministic mathematical model considering treatment at each infection stage of the co-infection model from which they showed that increase in treatment at each infection stage, reduced the co-infected individuals at each of the respective stages. Teklu and Rao \cite{b3} modified \cite{b1} to incorporate vaccination. The combined effort of treatment against HIV/AIDS-pneumonia co-infection and vaccination strategy against pneumonia showed a reduction in the co-epidemic burden.

The studies in \cite{b12,b1} and \cite{b3} do not specify the strain of pneumonia under consideration since different strains have different causative agents and hence different treatment regiments. The studies further do not incorporate the role played by PCP carriers in the transmission dynamics of HIV/AIDS and PCP co-infection. The interference in  the transmission chain of PCP among HIV/AIDS populace remains the only viable form to reduce the burden of the co-infection. Therefore, this study is designed with the aim of developing a \emph{Pneumocystis} pneumonia-HIV/AIDS co-infection deterministic model that will investigate the role played by PCP carriers in the transmission of PCP among people living with HIV and  further analyze the role of treatment of \emph{Pneumocystis} pneumonia at all contagious phases in the co-infection dynamics of \emph{Pneumocystis} pneumonia and HIV/AIDS. The developed model will guide on how properly the co-infection burden can be mitigated through treatment of carriers, infected and co-infected individuals and thus establishing a decisive directional strategy to policy makers.\\
This paper is organized as follows. In Section \ref{sec2}, we formulate a PCP and HIV/AIDS co-infection model with varying  force of infection. A qualitative analysis of the model and its constituent sub-models is evaluated in Section \ref{sec3}, numerical results are provided in Section \ref{sec4} and Section \ref{sec5} discusses and  concludes the study.
\section{Model Formulation}\label{sec2}
\subsection{Variables of the Model}
The model subdivides the human population into ten mutually-exclusive compartments, namely susceptible individuals at risk of contracting HIV or \emph{Pneumocystis} pneumonia $(S(t))$,  carrier individuals who carry the PCP  infection and can transmit the infection $(C(t))$, PCP-infected individuals who have active disease and are infectious $(I_P (t))$, PCP-recovered individuals who have been treated of the disease$(R(t))$, HIV infected individuals with no clinical symptoms of AIDS $(I_H(t))$, HIV-infected individuals with AIDS clinical symptoms $(I_A(t))$, individuals on treatment of HIV/AIDS $(I_T(t))$, HIV-infected individuals co-infected with PCP disease $(I_{HP}(t))$, HIV-infected individuals with AIDS symptoms co-infected with active PCP $(I_{AP}(t))$ and individuals on treatment of HIV/AIDS and PCP co-infection $(T(t))$.
The total population at time $t$, denoted by $N(t)$, is given by 
$$N(t)=S(t)+C(t)+I_P (t)+R (t)+I_H(t)+I_A(t)+I_T(t)+I_{HP}(t)+I_{AP}(t)+T(t).$$

In the formulation of the model, it is assumed that HIV and PCP infected classes are susceptible to each other, there is no permanent immunity for PCP treated individuals and all PCP carriers progress to active PCP class if left untreated. It is further assumed that  individuals in classes $I_T(t)$ and $T(t)$ do not participate in transmission of HIV, since antiretroviral therapy reduces the viral load and subsequently lowering the probability of transmission.

\subsection{Parameters of the Model and their Description}
Table \ref{t} below shows the parameters of the model and their description.
\begin{table}[h!]
\caption{Parameters used in the model formulation and their description}\label{t}
\begin{tabular}{ll}
Parameter & Definition \\
\hline
$\Lambda$ & Per capita recruitment rate into susceptible population \\
$\mu$ & Per capita natural mortality rate of individuals \\
$\nu_{P}$ & Per capita PCP induced mortality rate\\
$\nu_{A}$ & Per capita AIDS induced mortality rate\\
$\nu$& Per capita AIDS-PCP co-infection induced death rate\\
$\lambda_1 $ & Rate of progression from $I_H$ class to $I_A$ class in absence of treatment\\
$\lambda_2 $ & Rate of progression from $I_{HP}$ class to $I_{AP}$ class in absence of treatment\\
$\tau_1 $ & Rate of treatment of HIV infected individuals\\
$\tau_2 $ & Rate of treatment of AIDS patients \\
$\tau_3 $ & Rate of treatment of HIV and PCP co-infected individuals\\
$\tau_4 $ & Rate of treatment of AIDS and PCP co-infected individuals\\
$\beta$ &Rate at which PCP carriers develop symptoms\\
$\xi$ &Proportion of susceptible individuals that joins the PCP carriers \\
$\pi$ & Per capita recovery rate of PCP carriers\\
$\gamma$ & Rate at which treated PCP individuals become susceptible\\
$\epsilon$& Per capita recovery rate of PCP infected individuals\\
$\rho$& Probability that a contact is efficient enough to cause a PCP infection\\
$c$& Average rate of contact with pneumonia infected individuals\\
$\omega$& Transmission coefficient for the PCP carrier individuals\\
$\eta$ & Probability that one will acquire HIV upon contact with an infected individual\\
$k$ & Average contact rate with HIV infected individuals\\
\hline
\end{tabular}
\end{table}\\
\newpage
\subsection{Compartmental Diagram}
\begin{figure}[h!]
\begin{center}
\includegraphics[width=0.4\textwidth]{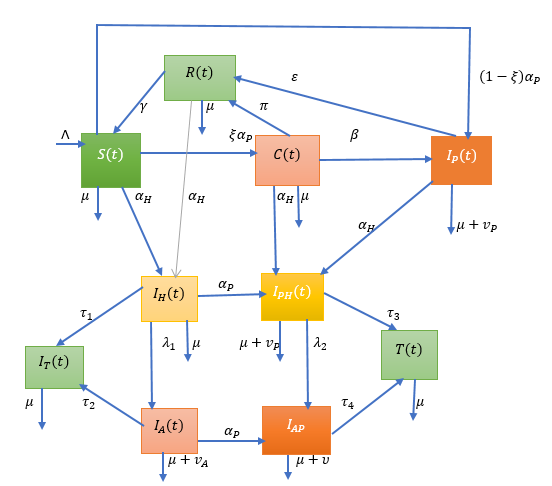}
\end{center}
\caption{A compartmental diagram for HIV/AIDS and PCP co-infection model.}\label{f1}
\end{figure}
\subsection{Equations of the Model}
Applying the assumptions, Figure \ref{f1}, definitions of variables and parameters in Table \ref{t}, the HIV/AIDS and PCP deterministic co-infection model  is obtained below.
\begin{align}
&\frac{dS}{dt}=\Lambda+\gamma R-\big(\alpha_P +\alpha_H +\mu\big) S,\nonumber\\
&\frac{dC}{dt}=\xi\alpha_P(t)S-\big(\alpha_H+\beta +\pi+\mu\big)C,\nonumber\\
&\frac{dI_P}{dt}=(1-\xi)\alpha_PS+\beta C-\big(\alpha_H+\epsilon+\mu+\nu_P\big)I_P,\nonumber\\
&\frac{dR}{dt}=\epsilon I_P+\pi C-\big(\alpha_H+\gamma+\mu\big)R,\nonumber\\
&\frac{dI_H}{dt}=\alpha_HS+\alpha_HR-\big(\alpha_P+\tau_1+\lambda_1+\mu\big)I_H,\nonumber\\
&\frac{dI_A}{dt}=\lambda_1 I_H-\big(\alpha_P+\tau_2+\mu+\nu_A\big) I_{A},\label{sys2}\\
&\frac{dI_T}{dt}=\tau_1I_H+\tau_2I_A-\mu I_T,\nonumber\\
&\frac{dI_{HP}}{dt}=\alpha_H\big(C+I_P\big)+\alpha_PI_H-\big(\lambda_2+\tau_3+\mu+\nu_P\big)I_{HP},\nonumber\\
&\frac{dI_{AP}}{dt}=\alpha_PI_A+\lambda_2I_{HP}-\big(\tau_4+\mu+\nu\big)I_{AP},\nonumber\\
&\frac{dT}{dt}=\tau_3I_{HP}+\tau_4I_{AP}-\mu T,\nonumber
\end{align}
where $t\geq 0$ with initial conditions $S(0)=S_0> 0,~C(0)=C_0\geq 0,~I_P(0)=I_{P0}\geq 0,~I_P(0)=I_{P0}\geq 0,~R(0)=R_{0}\geq 0,~I_H(0)=I_{H0}\geq 0,~I_A(0)=I_{A0}\geq 0,~I_T(0)=I_{T0}\geq 0,~I_{HP}(0)=I_{HP0}\geq 0,~I_{AP}(t)=I_{AP0}\geq 0,~T(0)=T_{0}\geq 0$.\\
The rate at which \emph{Pneumocystis} pneumonia spreads is defined as
$$\alpha_P(t)=\frac{\rho c(\omega C+I_P+\theta_1I_{HP}+\theta_2I_{AP})}{N(t)},$$ 
where $\theta_2 > \theta_1$ are modification parameters accounting for the assumed increased infectivity due to dual infection.\\
The rate at which HIV spreads  is defined as
$$\alpha_H(t)=\frac{\eta k(I_H+a_1I_{HP}+a_2I_A+a_3I_{AP})}{N(t)},$$ 
where  $a_3 > a_2 > a_1$, are modification parameters showing the infectious rate per class with the assumption that $I_{AP}$ individuals are more infectious of HIV than in individuals in classes $I_A$ an $I_{HP}$  due to high viral load \cite{b6,b14}.
\section{Analysis of the Model}\label{sec3}
\subsection{Basic Properties of the Model}
In this subsection, the basic properties of the solutions of model (\ref{sys2}) which are essential in the proofs of stability are studied.
\begin{lemma}\label{lem1}
The solutions $S(t),C(t),I_P(t),R(t),I_H(t),I_A(t),I_T(t),I_{HP}(t),I_{AP}(t)$ and $T(t)$ of system (\ref{sys2}) are non-negative for $t\geq 0.$
\end{lemma}
\begin{proof}
Let the initial values be non-negative, that is, $S(0)>0,C(0)>0,I_P(0)>0,R(0)>0,I_H(0)>0,I_A(0)>0,I_T(0)>0,I_{HP}(0)>0,I_{AP}(0)>0,T(0)>0$ , then for all $t>0$ we prove that $S(t)>0,C(t)>0,I_P(t)>0,R(t)>0,I_H(t)>0,I_A(t)>0,I_T(t)>0,I_{HP}(t)>0,I_{AP}(t)>0$ and $T(t)>0.$\\
By contradiction, assume that there exists a first time $t_1$ such that\\
 $S(t_1)=0, S'(t_1)<0$ and  $S(t)>0,C(t)>0,I_P(t)>0,R(t)>0,I_H(t)>0,I_A(t)>0,I_T(t)>0,I_{HP}(t)>0,I_{AP}(t)>0, T(t)>0 $ for $0<t<t_1$ or there exists a \\$t_2:C(t_2)=0, C'(t_2)<0$ $S(t)>0,C(t)>0,I_P(t)>0,R(t)>0,I_H(t)>0,I_A(t)>0,I_T(t)>0,I_{HP}(t)>0,I_{AP}(t)>0, T(t)>0 $ for $0<t<t_1$ or there exists a \\$t_3:I_P(t_3)=0, I_P'(t_3)<0$ $S(t)>0,C(t)>0,I_P(t)>0,R(t)>0,I_H(t)>0,I_A(t)>0,I_T(t)>0,I_{HP}(t)>0,I_{AP}(t)>0, T(t)>0 $ for $0<t<t_3$ or there exists a \\$t_4:R(t_4)=0, R'(t_4)<0$ $S(t)>0,C(t)>0,I_P(t)>0,R(t)>0,I_H(t)>0,I_A(t)>0,I_T(t)>0,I_{HP}(t)>0,I_{AP}(t)>0, T(t)>0 $ for $0<t<t_4$ or there exists a \\$t_5:I_H(t_5)=0, I_H'(t_5)<0$ $S(t)>0,C(t)>0,I_P(t)>0,R(t)>0,I_H(t)>0,I_A(t)>0,I_T(t)>0,I_{HP}(t)>0,I_{AP}(t)>0, T(t)>0 $ for $0<t<t_5$ or there exists a \\$t_6:I_A(t_6)=0, I_A'(t_6)<0$ $S(t)>0,C(t)>0,I_P(t)>0,R(t)>0,I_H(t)>0,I_A(t)>0,I_T(t)>0,I_{HP}(t)>0,I_{AP}(t)>0, T(t)>0 $ for $0<t<t_6$ or there exists a \\$t_7:I_T(t_7)=0, I_T'(t_7)<0$ $S(t)>0,C(t)>0,I_P(t)>0,R(t)>0,I_H(t)>0,I_A(t)>0,I_T(t)>0,I_{HP}(t)>0,I_{AP}(t)>0, T(t)>0 $ for $0<t<t_7$ or there exists a \\$t_8:I_{HP}(t_8)=0, I_{HP}'(t_8)<0$ $S(t)>0,C(t)>0,I_P(t)>0,R(t)>0,I_H(t)>0,I_A(t)>0,I_T(t)>0,I_{HP}(t)>0,I_{AP}(t)>0, T(t)>0 $ for $0<t<t_8$ or there exists a \\$t_9:I_{AP}(t_9)=0, I_{AP}'(t_9)<0$ $S(t)>0,C(t)>0,I_P(t)>0,R(t)>0,I_H(t)>0,I_A(t)>0,I_T(t)>0,I_{HP}(t)>0,I_{AP}(t)>0, T(t)>0 $ for $0<t<t_9$ or there exists a \\$t_{10}:T(t_{10})=0, C'(t_{10})<0$ $S(t)>0,C(t)>0,I_P(t)>0,R(t)>0,I_H(t)>0,I_A(t)>0,I_T(t)>0,I_{HP}(t)>0,I_{AP}(t)>0, T(t)>0 $ for $0<t<t_{10}$.\\
From system (\ref{sys2}), we see that
\begin{align}\frac{dS(t_1)}{dt}&=\Lambda+\gamma R(t_1)>0, \nonumber\\ \frac{dC(t_2)}{dt}&=\xi\alpha_P(t_2)S(t_2)>0,\nonumber\\ \frac{dI_P(t_3)}{dt}&=(1-\xi)\alpha_PS(t_3)+\beta C(t_3)>0,\nonumber\\\frac{dR(t_4)}{dt}&=\epsilon I_P(t_4)+\pi C(t_4)>0,\nonumber\\\frac{dI_H(t_5)}{dt}&=\alpha_HS(t_5)+\alpha_HR(t_5)>0,\nonumber\\\frac{dI_A(t_6)}{dt}&=\lambda_1 I_H(t_6)>0,\nonumber\\\frac{dI_T(t_7)}{dt}&=\tau_1I_H(t_7)+\tau_2I_A(t_7)>0,\nonumber\\\frac{dI_{HP}(t_8)}{dt}&=\alpha_H(t_8)\big(C(t_8)+I_P(t_8)\big)+\alpha_PI_H(t_8)>0,\nonumber\\\frac{dI_{AP}(t_9)}{dt}&=\alpha_P(t_9)I_A(t_9)+\lambda_2I_{HP}(t_9)>0,\nonumber\\\frac{dT(t_{10})}{dt}&=\tau_3I_{HP}(t_{10})+\tau_4I_{AP}(t_{10})>0,\nonumber\end{align} which leads to a contradiction and consequently, $S(t),C(t),I_P(t),R(t),I_H(t),I_A(t),I_T(t),I_{HP}(t),I_{AP}(t)$ and $T(t)$ remain positive. Therefore, the solutions of system (\ref{sys2}) are non-negative for $t\geq 0$.
\end{proof}
\begin{lemma}\label{lem2}
The \emph{Pneumocystis} pneumonia and HIV/AIDS co-infection model (\ref{sys2}) is mathematically and epidemiologically well-posed.
\end{lemma}
\begin{proof}
By Lemma \ref{lem1}, $N(0)=N_0> 0.$
Adding the equations of system (\ref{sys2}) gives 
\begin{equation}\frac{dN(t)}{dt}=\Lambda-\mu N(t)-\nu_P\big(I_P(t)+I_{HP}(t)\big)-\nu_AI_A(t)-\nu I_{AP}(t).\label{sys2a}\end{equation} Since $I_P(t)\geq 0, I_A(t)\geq 0, I_{HP}(t)\geq 0$ and $I_{AP}(t)\geq 0 $ for all $t\geq 0$, thus (\ref{sys2a}) gives the inequality
\begin{equation}\frac{dN(t)}{dt}+\mu N(t)\leq \Lambda.\label{sys2b}\end{equation}
With initial condition $N(0)=N_0$, integrate (\ref{sys2b}) to get 
\begin{equation}N(t)\leq \frac{\Lambda}{\mu}+\Big(N_0-\frac{\Lambda}{\mu}\Big)e^{-\mu t}.\label{sys2c}\end{equation}
As $t\to\infty$, we have $0\leq N(t)\leq \frac{\Lambda}{\mu}$.
Therefore, every feasible solution of model (\ref{sys2}) that starts in region\\ $\Omega=\Big\{\big(S(t),C(t),I_P(t),R(t),I_H(t),I_A(t),I_T(t),I_{HP}(t),I_{AP}(t),T(t)\big)\in\mathbb{R}_+^{10}:N(t)\leq \Lambda/\mu\Big\}$ remains in this region for all values of $t\geq 0$ which implies that $\Omega$ is positively invariant and attracting. Hence the HIV/AIDS-PCP co-infection model (\ref{sys2}) is mathematically and epidemiologically well-posed.
\end{proof}
It is crucial to analyze the HIV/AIDS and PCP  sub-models first and there after analyze the full co-infection model (\ref{sys2}). This is done in order to have an insight on the transmission dynamics of the respective sub-models that are initial stages of a potential interaction between the two diseases. 
\subsection{Analysis of HIV/AIDS Sub-model}
By setting $C(t)=I_P(t)=R(t)=I_{HP}(t)=I_{AP}(t)=T(t)=0,$  the following HIV/AIDS sub-model is obtained:
\begin{align}
&\frac{dS}{dt}=\Lambda-(\alpha_H+\mu)S,\nonumber\\
&\frac{dI_H}{dt}=\alpha_H S-(\tau_1+\lambda_1+\mu)I_H,\nonumber\\
&\frac{dI_A}{dt}=\lambda_1I_H-(\tau_2+\nu_A+\mu)I_A,\label{sys3a}\\
&\frac{dI_T}{dt}=\tau_1I_H+\tau_2I_A-\mu I_T,\nonumber
\end{align}
with 
\begin{equation}
\alpha_H=\eta k\frac{(I_H+a_2I_A)}{N_H}, \label{sys3b}
\end{equation}
so that 
\begin{align}
\frac{dN_H}{dt}=\Lambda-\mu N-\nu_AI_A.\nonumber
\end{align}
\newtheorem{Theorem}{Theorem}

\subsubsection{Disease-Free Equilibrium of the HIV/AIDS Sub-model}
The equilibria of system (\ref{sys3a}) are obtained by setting the right hand side of system (\ref{sys3a}) equal to zero.
The disease-free equilibrium of the model describes the model in absence of the disease or infection. Thus with  $I_{H}=I_{A}=0$ and $I_T=0$, we have  $S=\Lambda/\mu.$
Therefore, the HIV/AIDS-free equilibrium of (\ref{sys3a}) is given as $E_{0H}=\big(\Lambda/\mu,0,0,0\big).$                                                                                                                              
\subsubsection{Computation of the Basic Reproduction Number of the HIV/AIDS Sub-model}\label{subsec1.22}
Rephrasing the definition by Diekmann et al.  \cite{b15}, we define the HIV/AIDS basic reproduction number as the average number of secondary HIV infections caused by a single HIV infectious individual during his or her entire period of infectiousness. We use the next-generation matrix method as applied in \cite{b23,b26,b24,b16} to determine the basic reproduction number $\mathcal{R}_{0H}$ of (\ref{sys3a}). Let $\mathcal{F}$ denote the  matrix of the new infection terms and $\mathcal{V}$ be the matrix of the remaining transfer terms in system (\ref{sys3a}). Then we have
$$\mathcal{F}=\begin{bmatrix}\alpha_H S \\ 0 \end{bmatrix}, \mathcal{V}=\begin{bmatrix}(\tau_1+\lambda_1+\mu)I_H\\-\lambda_1I_H(t)+(\tau_2+\nu_A+\mu)I_A(t)\end{bmatrix}.$$  
We obtain the matrices $F$ and $V$ by finding the Jacobian matrices of $\mathcal{F}$ and $\mathcal{V}$ evaluated at the disease-free equilibrium  respectively to obtain\\
$$F=\begin{bmatrix}\eta k & \eta k a_2 \\ 0& 0\end{bmatrix}, V=\begin{bmatrix}q_1 & 0 \\ -\lambda & q_2\end{bmatrix},$$ where $q_1=\tau_1+\lambda_1+\mu$ and $q_2=\tau_2+\nu_A+\mu$.\\
Thus, the basic reproduction number is given by  $$\mathcal{R}_{0H}=\mathcal{K}(FV^{-1})=\frac{\eta k}{q_1q_2}\big(q_2+a_2\lambda_1\big), $$
where $\mathcal{K}$ represents the spectral radius of $FV^{-1}$.
\subsubsection{Endemic Equilibrium Point of the HIV/AIDS Sub-model}
The endemic equilibrium point denoted by $E_{eH}$ is defined as a steady state solution for system (\ref{sys3a}) which occurs when there is persistence of the HIV/AIDS in a community.
Equating the right hand side equal to zero of  system (\ref{sys3a}) and solving for $S, I_H,I_A$ and $I_T$ yields 
  $ E_{eH}=(S_e, I_{He}, I_{Ae},I_{Te} )$, where 
\begin{align}
& S_e=\frac{N_Hq_1q_2}{\eta k(q_2+a_2\lambda_1)},\nonumber\\
& I_{He}=\frac{\mu N_Hq_1q_2-\Lambda\eta k(q_2+a_2\lambda_1)}{q_1\eta k(q_2+a_2\lambda_1)},\nonumber\\
& I_{Ae}=\frac{\lambda_1\big(\mu N_Hq_1q_2-\Lambda\eta k(q_2+a_2\lambda_1)\big)}{q_1q_2\eta k(q_2+a_2\lambda_1)},\nonumber\\
& I_{Te}=\frac{\mu N_Hq_1q_2-\Lambda\eta k(q_2+a_2\lambda_1)}{\mu q_1\eta k(q_2+a_2\lambda_1)}\Big(\tau_1+\frac{\tau_2\lambda_1}{q_2}\Big),\nonumber
\end{align}
where $N_H=\frac{\Lambda\eta k(q_1q_2+\nu_A\lambda_1)(q_2+a_2\lambda_1)}{q_1q_2\big(\mu \eta k(q_2+a_2\lambda_1)+\nu_A\lambda_1\big)}.$
\begin{lemma}
The HIV/AIDS sub-model (\ref{sys3a}) has a unique endemic equilibrium point if $\mathcal{R}_{0H}>1$.
\end{lemma}
\begin{proof}
If the disease is endemic in the community, then $\frac{dI_H}{dt}>0$ and $\frac{dI_A}{dt}>0$ that is 
\begin{align}
&\eta k(I_H+a_2I_A)\frac{S}{N}-q_1I_H>0,\label{sys3m}\\
&\lambda_1I_H-q_2I_A>0.\label{sys3n}
\end{align}
From (\ref{sys3m}) and using the fact that $\frac{S}{N}<1$, we have 
\begin{equation}I_H<\frac{\eta k a_2I_A}{q_1-\eta k}.\label{sys3o}\end{equation}
From (\ref{sys3n}), we have \begin{equation}I_A<\frac{\lambda_1 I_H}{q_2}.\label{sys3p}\end{equation}
Inequalities (\ref{sys3o}) and (\ref{sys3p}) together imply
\begin{align} &1<\frac{\eta k}{q_1q_2}(q_2+a_2\lambda_1)=\mathcal{R}_{0H}.\nonumber\end{align}
Thus, a unique endemic equilibrium point $E_{eH}$ exists when $\mathcal{R}_{0H}>1.$
\end{proof}
\subsubsection{Local and Global Stability of HIV/AIDS Sub-Model Disease-Free Equilibrium}
\begin{lemma}
The HIV/AIDS-free equilibrium $E_{0H}$ is locally asymptotically stable if $\mathcal{R}_{0H}<1$ and unstable otherwise.
\end{lemma}
\begin{proof}
For $E_{0H}$ to be locally asymptotically stable, the Jacobian matrix $J_{E_{0H}}$ of sub-model (\ref{sys3a})  should have  negative eigenvalues or equivalently a negative trace and a positive determinant.
\begin{equation}
J_{E_{0H}}=\begin{bmatrix}-\mu&-\eta k&-\eta ka_2&0\\ 0&\eta k-q_1&\eta k a_2&0\\ 0&\lambda_1&-q_2&0\\ 0&\tau_1&\tau_2&-\mu\end{bmatrix}.\label{sys3x}
\end{equation}
The first and fourth columns of the Jacobian matrix (\ref{sys3x}) clearly show that $-\mu$ is a repeated negative eigenvalue. The other two eigenvalues can be obtained by reducing the Jacobian matrix (\ref{sys3x})  into a $2\times 2$ matrix given by
\begin{equation}J_{E_{0H}}^*=\begin{bmatrix}\eta k-q_1&\eta ka_2\\\lambda_1&-q_2\end{bmatrix}.\label{sys3y}\end{equation}
We now employ the trace determinant strategy on (\ref{sys3y}) such that \begin{align}&tr(J_{E_{0H}})=\eta k-(q_1+q_2)<0~\text{if}~\eta k<q_1+q_2.\label{sys3z}\end{align}
From $ \mathcal{R}_{0H}=\frac{\eta k}{q_1q_2}\big(q_2+a_2\lambda_1\big)$, it is easy to see that (\ref{sys3z}) gives \begin{align}&\frac{\eta k}{q_1q_2}\big(q_2+a_2\lambda_1\big)<\frac{q_1+q_2}{q_1q_2}\big(q_2+a_2\lambda_1\big),\nonumber\\ &\mathcal{R}_{0H}<\frac{q_1+q_2}{q_1q_2}\big(q_2+a_2\lambda_1\big).\label{sys3aa}\end{align}
Thus $tr(J_{E_{0H}})<0$ if $\mathcal{R}_{0H}<\frac{q_1+q_2}{q_1q_2}\big(q_2+a_2\lambda_1\big).$\\
We now consider $\det(J_{E_{0H}})=-\eta k(q_2+\lambda_1 a_2)+q_1q_2>0~\text{if}~\eta k(q_2+\lambda_1 a_2)<q_1q_2.$
Therefore,  \begin{equation}\mathcal{R}_{0H}=\frac{\eta k}{q_1q_2}\big(q_2+a_2\lambda_1\big)<1\implies \det(J_{E_{0H}})>1~\text{if}~\mathcal{R}_{0H}<1.\label{sys3ab}\end{equation}
Thus, $E_{0H}$ is locally asymptotically stable if and only if the inequalities (\ref{sys3aa}) and (\ref{sys3ab}) hold.
\end{proof}
\begin{Theorem}
The disease-free equilibrium  $E_{0H}$ of system (\ref{sys3a}) is globally asymptotically stable if $\mathcal{R}_{0H}<1$ and unstable otherwise. The disease free equilibrium $E_{0H}$ is the only equilibrium when $\mathcal{R}_{0H}\leq 1$.
\end{Theorem}
\begin{proof}
Let \begin{equation}W=\psi_1I_H+\psi_2 I_A,\label{sys3ac}\end{equation} be the Lyapunov function which involves individuals who contribute to HIV/AIDS in the population where  $\psi_1$ and $\psi_2$ are non-negative constants. The time derivative of the Lyapunov function (\ref{sys3ac}) is given by 
\begin{align}\frac{dW}{dt}&=\psi_1\big(\alpha_HS-q_1I_H\big)+\psi_2\big(\lambda_1-q_2I_A\big),\nonumber\\&\leq\big((\eta k-q_1)\psi_1+\psi_2\lambda_1\big)I_H+\big(\psi_1\eta ka_2-\psi_2q_2\big)I_A.\label{sys3ad}\end{align}
Fixing $\psi_1>0$ and setting $\psi_2=\frac{1}{q_2}\eta ka_2\psi_1$, we obtain
\begin{align}\frac{dW}{dt}&\leq q_1\psi_1\Big(\frac{\eta k(q_2+\lambda_1a_2)}{q_1q_2}-1\Big)I_H+\frac{q_1q_2\psi_1}{\lambda_1}\Big(\frac{\eta k(q_2+a_2\lambda_1)}{q_1q_2}-1\Big)I_A,\nonumber\\
&=\Big(q_1\psi_1I_H+\frac{q_1q_2\psi_1}{\lambda_1}I_A\Big)\Big(\mathcal{R}_{0H}-1\Big).\nonumber\end{align}
Thus, $\frac{dW}{dt}\leq 0$ when $\mathcal{R}_{0H} \leq 1$. Furthermore, $\frac{dW}{dt}=0$ if and only if either $I_H=I_A=0$ or $\mathcal{R}_{0H}=1$. In either case, the largest compact invariant subset of  $\Omega_1=\{S(t),I_H(t),I_A(t),I_T(t)\in \mathbb{R}^4_+:\frac{dW}{dt}=0\}$ is the singleton $E_{0H}$. By Lasalle's Invariance Principle \cite{b11,b17},  $E_{0H}$ is globally stable in $\mathbb{R}^4_+$ provided $\mathcal{R}_{0H}\leq 1.$
\end{proof}
\subsubsection{Local and Global Stability of the HIV/AIDS Endemic Equilibrium}
\begin{Theorem}
The endemic equilibrium point $E_{eH}$ is locally asymptotically stable if $\mathcal{R}_{0H}>1$, otherwise it is unstable.
\begin{proof}
To show the local stability of the endemic equilibrium point, the method of Routh-Hurwitz stability criteria is employed \cite{b34,b24}. The Jacobian matrix of the system (\ref{sys3a}) evaluated at the endemic equilibrium point $E_{eH}$ is 
\begin{equation}
J_{E_{eH}}=\begin{bmatrix}-B_1&-B_3&-B_4&0\\\alpha_H^*&B_2&B_5&0\\0&\lambda_1&-q_2&0\\0&\tau_1&\tau_2&-\mu\end{bmatrix},\label{sys3ae}
\end{equation}
where $B_1=(\alpha_H^*+\mu),B_2=\eta k/\mathcal{R}_{0H}-q_1, B_3=\eta k/\mathcal{R}_{0H},B_4=B_5=\eta ka_2/\mathcal{R}_{0H}, q_1=(\tau_1+\lambda_1+\mu), q_2=(\tau_2+\nu_A+\mu)$ and $\alpha_H^*$ is the force of HIV infection evaluated at the disease endemic equilibrium point.\\
The characteristic equation of Jacobian (\ref{sys3ae}) is given by 
\begin{align}P(\lambda)=&\det\begin{bmatrix}\lambda+B_1&B_3&B_4&0\\-\alpha_H^*&\lambda-B_2&-B_5&0\\0&-\lambda_1&\lambda+q_2&0\\0&-\tau_1&-\tau_2&\lambda+\mu\end{bmatrix},\nonumber\\=&(\lambda+\mu)\Big(\lambda^3+(B_1+q_2-B_2)\lambda^2+\big(B_1(q_2-B_2)+B_3\alpha_H^*-B_2q_2-\lambda_1B_5\big)\lambda\nonumber\\&+(B_3q_2+B_4\lambda_1)\alpha_H^*-(B_2q_2+\lambda_1B_5)B_1\Big).\label{sys3af}\end{align}
From characteristic equation (\ref{sys3af}), we have $\lambda=-\mu<0$ and \begin{equation}\lambda^3+a_1\lambda^2+a_2\lambda+a_3=0,\label{sys3ag}\end{equation} where 
\begin{align}
&a_1=B_1+q_1-B_2=\alpha_H^*+q_1+q_2+\mu-\eta k/\mathcal{R}_{0H}>0~\text{if}~\mathcal{R}_{0H}>1,\nonumber\\
&a_2=B_1(q_2-B_2)+B_3\alpha_H^*-B_2q_2-\lambda_1 B_5=(\alpha_H^*+\mu)(q_1+q_2-\eta k/\mathcal{R}_{0H})+q_1q_2/\mathcal{R}_{0H}(\mathcal{R}_{0H}-1)>0~\text{if}~\mathcal{R}_{0H}>1,\nonumber\\
&a_3=\alpha_H^*(B_3q_2+B_4\lambda_1)-B_1(B_2q_2+B_5\lambda_1)=\eta k/\mathcal{R}_{0H}(\alpha_H^*q_2+a_2\lambda_1)+(\alpha_H^*+\mu)/\mathcal{R}_{0H}(\mathcal{R}_{0H}-1)>0~\text{if}~\mathcal{R}_{0H}>1.\nonumber
\end{align}
It has been estalished that for $\mathcal{R}_{0H}>1$, $a_1,a_2$ ,$a_3$ are positive and we now compute $a_1a_2-a_3$  to get 
$$a_1a_2-a_3=\big(\alpha_H^*+q_1+q_2+\mu-\frac{\eta k}{\mathcal{R}_{0H}}\big)\Big((\alpha_H^*+\mu)(q_1+q_2-\frac{\eta k}{\mathcal{R}_{0H}}\Big)+\frac{q_1q_2}{\mathcal{R}_{0H}}(\mathcal{R}_{0H}-1)\Big(q_1+q_2-\frac{\eta k}{\mathcal{R}_{0H}}\Big)>0,~\text{if}~\mathcal{R}_{0H}>1.$$
The Routh-Hurwitz conditions $(a_1>0,a_2>0,a_3>0,a_1a_2>a_3)$  for characteristic equation (\ref{sys3ag}) are satisfied whenever $\mathcal{R}_{0H}>1$. Hence, the endemic equilibrium point of system (\ref{sys3a}) is locally asymptotically stable.
\end{proof}
\end{Theorem}
\begin{Theorem}
If  $\mathcal{R}_{0H}>1$, the endemic equilibrium point $E_{eH}$ of the model (\ref{sys3a}) is globally asymptotically stable.
\end{Theorem}
\begin{proof}
To establish the global stability of the endemic equilibrium $E_{eH}$, the following Lyapunov function   $\mathcal{W}(S,I_H,I_A,I_T)$ is constructed as
\begin{align}\mathcal{W}=\Big(S-S_e+S_e\ln\frac{S_e}{S}\Big)+\Big(I_H-I_{He}+I_{He}\ln\frac{I_{He}}{I_H}\Big)+\Big(I_A-I_{Ae}+I_{Ae}\ln\frac{I_{Ae}}{I_A}\Big)+\Big(I_T-I_{Te}+I_{Te}\ln\frac{I_{Te}}{I_T}\Big).\label{sys3ah}\end{align}
It is clear that $\mathcal{W}(S_e,I_{He},I_{Ae},I_{Te})=0$, and $\mathcal{W}>0$ otherwise. Moreover, $W$ is radially unbounded in $\mathbb{R}_+^4$.
We now seek to determine the sign of $d\mathcal{W}/dt$ by differentiating (\ref{sys3ah}) with respect to $t$ as follows
\begin{align}\frac{d\mathcal{W}}{dt}=\Big(1-\frac{S_e}{S}\Big)\frac{dS}{dt}+\Big(1-\frac{I_{He}}{I_H}\Big)\frac{dI_H}{dt}+\Big(1-\frac{I_{Ae}}{I_A}\Big)\frac{dI_A}{dt}+\Big(1-\frac{I_{Te}}{I_T}\Big)\frac{dI_T}{dt}.\label{sys3ai}\end{align}
Substituting for $dS/dt,dI_H/dt,dI_A/dt$ and $dI_T/dt$ into (\ref{sys3ai}) gives
\begin{align}
\frac{d\mathcal{W}}{dt}=&\Big(1-\frac{S_e}{S}\Big)\big(\Lambda-(\alpha_H+\mu)\big)+\Big(1-\frac{I_{He}}{I_H}\Big)\big(\alpha_H-q_1I_H\big)+\Big(1-\frac{I_{Ae}}{I_A}\Big)\big(\lambda_1I_H-q_2I_A\big)\nonumber\\
&+\Big(1-\frac{I_{Te}}{I_T}\Big)\big(\tau_1I_H+\tau_2I_A-\mu I_T\big),\nonumber\\
=&\Lambda+(\alpha_H+\mu)S_e+\alpha_HS+q_1I_{He}+(\lambda_1+\tau_1)I_H +q_2I_{Ae}+\tau_2I_A+\mu I_{Te}-\Big(\Lambda\frac{S_e}{S}\nonumber\\
&+\big(\alpha_H(1+\frac{I_{He}}{I_H})+\mu\big)S+\big(q_1+\lambda_1\frac{I_{Ae}}{I_A}+\tau_1\frac{I_{Te}}{I_T}\big)I_H+\big(q_2+\tau_2\frac{I_{Te}}{I_T}\big)I_A+\mu I_T\Big),\nonumber\\
\frac{d\mathcal{W}}{dt}=&P-Q.\label{sys3aj}
\end{align}
From (\ref{sys3aj}), we have 
$P=\Lambda+(\alpha_H+\mu)S_e+\alpha_HS+q_1I_{He}+(\lambda_1+\tau_1)I_H +q_2I_{Ae}+\tau_2I_A+\mu I_{Te}$ and \\$Q=\Lambda\frac{S_e}{S}+\big(\alpha_H(1+\frac{I_{He}}{I_H})+\mu\big)S+\big(q_1+\lambda_1\frac{I_{Ae}}{I_A}+\tau_1\frac{I_{Te}}{I_T}\big)I_H+\big(q_2+\tau_2\frac{I_{Te}}{I_T}\big)I_A+\mu I_T$.\\
Therefore from (\ref{sys3aj}), if $P<Q$, then $\frac{d\mathcal{W}}{dt}$ will be negative definite, implying that $\frac{d\mathcal{W}}{dt}<0.$ It is also clear that $\frac{d\mathcal{W}}{dt}=0$ if and only if $S=S_e,I_H=I_{He},I_A=I_{Ae}$ and $I_T=I_{Te}$.\\
Therefore, the largest compact invariant set in $\{(S,I_H,I_A,I_T)\in\Omega_1:\frac{d\mathcal{W}}{dt}=0\}$ is the singleton $\{E_{eH}\}$ where $E_{eH}$ is the endemic equilibrium point of system (\ref{sys3a}). By La Salle's invariant principle \cite{b11,k,b17}, it implies that $E_{eH}$ is globally asymptotically stable in $\Omega_1$ if $P\leq Q$, which holds if and only if $\mathcal{R}_{0H}> 1$.
\end{proof}
\subsection{Analysis of \emph{Pneumocystis} Pneumonia Sub-model}
We obtain the PCP sub-model when $I_H(t)=I_A(t)=I_{T}(t)=I_{HP}(t)=I_{AP}=T(t)=0$ in system \ref{sys2}, which gives
\begin{align}
&\frac{dS}{dt}=\Lambda+\gamma R-(\alpha_P+\mu)S,\nonumber\\
&\frac{dC}{dt}=\xi \alpha_PS-(\beta+\pi+\mu)C,\nonumber\\
&\frac{dI_P}{dt}=(1-\xi)\alpha_P S+\beta C-(\epsilon+\nu_P+\mu)I_P,\label{sys4a}\\
&\frac{dR}{dt}=\pi C+\epsilon I_P-(\gamma+\mu)R,\nonumber\\
&N_P=S+C+I_P+R.\nonumber
\end{align}

\subsubsection{Disease-Free Equilibrium of the PCP Sub-model}
The disease free equilibrium point of  model (\ref{sys4a}) describes the model in absence of PCP that is the carrier and infected classes are zero which gives, $C=I_{P}=0, R=0$ and $S=\Lambda/{\mu}.$
Therefore, the disease-free equilibrium point denoted by $E_{0P}=(S,C,I_{P},R )$ is ${E}_{0P}=\Big(\Lambda/\mu,0,0,0\Big).$
\subsubsection{\emph{Pneumocystis} Pneumonia Basic Reproduction Number}
Following a similar  procedure in sub-subsection \ref{subsec1.22}, the PCP basic reproduction number is computed and obtained as 

$$ \mathcal{R}_{0P}=\frac{\rho c}{k_1k_2}\Big(\xi(k_2\omega+\beta)+(1-\xi)k_1\Big). $$
\subsubsection{Endemic Equilibrium Point}
The endemic equilibrium point denoted by $E_{eP}$ is defined as a steady state solution for system (\ref{sys4a}) and it occurs when there is a persistence of \emph{Pneumocystis} Pneumonia. Hence, $E_{eP}=( S_e, C_e, I_{Pe}, R_e )$ can be determined by solving system (\ref{sys4a}) with the right hand side equal to zero, from which we obtain,

\begin{eqnarray} 
 S_e&=&\frac{k_1k_2 N_P}{\rho c\big((1-\xi)k_1+\xi(\omega k_2+\beta)\big)},\nonumber \\
C_e&=&\frac{\xi k_2k_3\Big(\mu k_1k_2N_P-\Lambda\rho c\big((1-\xi)k_1+\xi(\omega k_2+\beta)\big)\Big)}{A},\nonumber\\
I_{Pe}&= &\frac{k_3\Big((1-\xi)k_1+\xi\beta\Big)\Big(\mu k_1k_2N_P-\Lambda\rho c\big((1-\xi)k_1+\xi(\omega k_2+\beta)\big)\Big)}{A},\nonumber\\
R_e&=&\frac{\Big(\mu k_1k_2N_P-\Lambda\rho c\big((1-\xi)k_1+\xi(\omega k_2+\beta)\big)\Big)}{A}\Big(\pi\xi k_2+\epsilon\big((1-\xi)k_1+\xi\beta\big)\Big),\nonumber
\end{eqnarray}
together with $A=\rho c\Big((1-\xi)k_1+\xi\beta\Big)\Big(\xi k_2\gamma \pi-k_1k_2k_3+\gamma\epsilon\big((1-\xi)k_1+\xi\beta\big)\Big)$ and \\$N_P=\frac{\Lambda\Big(A+\nu_P\rho ck_3\big(1-\xi)k_1+\xi\beta\big)\big((1-\xi)k_1+\xi(\omega k_2+\beta)\big)\Big)}{\mu A-\nu_P\mu k_1k_2k_3\big((1-\xi)k_1+\xi\beta\big)}.$
\begin{lemma} 
For $\mathcal{R}_{0P}>1$, the system (\ref{sys4a}) has a unique endemic equilibrium point $E_e$ and no endemic equilibrium otherwise.
\end{lemma}
\begin{proof}
If the disease persists in the community, then $\frac{dC}{dt}>0$ and $\frac{dI_P}{dt}>0$, that is, 
\begin{eqnarray}
\xi\alpha_P S-k_1C=\xi\rho c(\omega C+I_P)\frac{S}{N}-k_1C&>&0,\label{e1}\\
(1-\xi)\alpha_P S+\beta-k_2I_P=(1-\xi)\rho c(\omega C+I_P)\frac{S}{N}+\beta C-k_2I_P&>&0.\label{e2}
\end{eqnarray}
From inequality (\ref{e1}) and applying the fact that $S/N\leq 1$, we have
\begin{equation}C< \frac{\xi\rho cI_P}{k_1-\xi \rho c \omega},\label{e3}\end{equation}
Inequalities (\ref{e2}) and (\ref{e3}), together imply
\begin{align}
\mathcal{R}_{0P}=\frac{\rho c}{k_1k_2}\Big((1-\xi)k_1+\xi(\beta+k_2\omega)\Big)&>1. \nonumber 
\end{align}
Thus, a unique endemic equilibrium point exists when $\mathcal{R}_{0P}>1$.
\end{proof}
\subsubsection{Local and Global  Stability of PCP  Free Equilibrium}
\begin{Theorem} The disease-free equilibrium point ${E}_{0P}$ of system (\ref{sys4a}) is locally asymptotically stable whenever $\mathcal{R}_{0P} < 1$ and unstable whenever $\mathcal{R}_{0P} > 1$.\end{Theorem}
\begin{proof}
The Jacobian matrix of sub-model system (\ref{sys4a}) evaluated at the disease-free equilibrium point ${E}_{0P}$ is given as  
\begin{equation}J_{({E}_{0P})}=\begin{bmatrix}-\mu & -\rho c\omega& -\rho c& \gamma\\ 0 & \xi\rho c\omega-k_1& \xi\rho c & 0 \\0 & (1-\xi)\rho c\omega+\beta & (1-\xi)\rho c-k_2 & 0 \\0 & \pi & \epsilon & -k_3\end {bmatrix}.\label{sys4aa}\end{equation}
Thus, from the Jacobian matrix (\ref{sys4aa}), we obtain the characteristic equation;
\begin{align}P(\lambda)=(\lambda+\mu)(\lambda+k_3)\Big(\lambda^2+a_1\lambda+a_0\Big),\label{sys4ab}\end{align}
where $a_1=k_1+k_2-\rho c((1-\xi)+\xi\omega)>0$ and $a_0=k_1k_2(1-\mathcal{R}_{0P})>0$ if $\mathcal{R}_{0P}<1$.\\
Clearly the roots of (\ref{sys4ab}) $\lambda=-\mu $, $\lambda=-k_3$ are negative and by Routh-Hurwitz criterion, the characteristic polynomial $\lambda^2+a_1\lambda+a_0$ has negative eigenvalues if $a_0>0,a_1>0$ which is satisfied when $\mathcal{R}_{0P}<1$.
Therefore, since all the eigenvalues of (\ref{sys4aa}) are negative whenever $\mathcal{R}_{0P}<1$, the disease-free equilibrium point of system (\ref{sys4a}) is locally asymptotically stable.
\end{proof}

We now deploy the method described by Castillo-Chavez \emph{et al.} \cite{b18} to study the global asymptotic stability of the disease-free equilibrium.  Re-write the model system (\ref{sys4a}) as
\begin{align}
&\frac{dX_1}{dt}=\mathcal{F}(X_1,X_2),\nonumber\\
&\frac{dX_2}{dt}=\mathcal{G}(X_1,X_2),\label{sys4ac}\\
&\mathcal{G}(X_1,0)=0,\nonumber
\end{align}
where $X_1=(S,R)$ represents the uninfected population and $X_2=(C,I_P)$ represents the infected population. The disease-free equilibrium of the model is denoted by ${E}_{0P}=(X^*,0),~\text{where} ~X^*=\big(\Lambda/{\mu},0,0,0\big).$
\begin{Theorem} 
The disease free-equilibrium point  ${E}_{0P}$ is globally asymptotically stable if  $\mathcal{R}_{0P}<1$ and the following conditions should hold;
\begin{enumerate}[(T1)]
\item for $\frac{dX_1}{dt}=\mathcal{F}(X_1,0), X^*$ is globally asymptotically stable,
\item $\mathcal{G}(X_1,X_2)=AX_2-\tilde{\mathcal{G}}(X_1,X_2),$ $\tilde{\mathcal{G}}(X_1,X_2)\geq 0$ for $(X_1,X_2)\in\Omega_2$ and $A=\frac{\partial{\mathcal{G}(X_1^*,0)}}{\partial{X_2}}$ is an $M-$matrix (the off diagonal elements are non-negative) and $\Omega_2$ is the invariant region.
\end{enumerate}
\end{Theorem}

\begin{proof}
From system (\ref{sys4a}), it follows that 
\begin{align}
&\mathcal{F}(X_1,X_2)=\begin{bmatrix}\Lambda +\gamma R-(\alpha_P+\mu)S\\\pi C+\epsilon I_P-(\gamma+\mu)R  \end{bmatrix},\nonumber\\
&\mathcal{G}(X_1,X_2)=\begin{bmatrix} \xi\alpha_P S-k_1C\\(1-\xi)\alpha_PS+\beta C-k_2I_P\end{bmatrix}.\label{sys4ad}
\end{align}
Consider
$\mathcal{F}(X_1,0)=\begin{bmatrix}\Lambda-\mu S\\0\end{bmatrix},$ and as $t\to\infty$, its observed that $X_1\to E_{0P}$  that is,  $S\to \Lambda/\mu$. Thus, there is convergence in $\Omega_2$ implying that (T1) holds.\\
Now
\begin{align}
&A=\frac{\partial{\mathcal{G}(X_1^*,0)}}{\partial{X_2}}=\begin{bmatrix}-k_1+\xi\rho c\omega&\xi\rho c\\\beta+(1-\xi)\rho c\omega&-k_2+(1-\xi)\rho c\end{bmatrix},\label{sys4ae}
\end{align}
and  \begin{align}\tilde{\mathcal{G}}(X_1,X_2)&=AX_2-\mathcal{G}(X_1,X_2)=\begin{bmatrix}\xi\rho c\Big(1-\frac{S}{N}\Big)\big(\omega C+I_P\big)\\(1-\xi)\rho c\Big(1-\frac{S}{N}\Big)\big(\omega C+I_P\big)\end{bmatrix}.\label{sys4af}\end{align}
Clearly from (\ref{sys4ae}), $A$ is an $M-$matrix and from (\ref{sys4af}), $\tilde{\mathcal{G}}(X_1,X_2)\geq 0$ since $\frac{S}{N}\leq 1$ implying that (T2) holds. Since both (T1) and (T2) hold, then $E_{0P}$ is globally asymptotically stable.
\end{proof}
\subsubsection{Global stability of PCP Endemic Equilibrium}
\begin{Theorem}: If $\mathcal{R}_{0P}>1$, the endemic equilibrium ${E}_{eP}$ of the system (\ref{sys4a}) is globally asymptotically stable.\end{Theorem}
\begin{proof}
Systematically, we define an appropriate  Lyapunov function $L(S, C,I_P,R)$  such that
\begin{equation}L=\Big(S-S_e+S_e \ln\frac{S_e}{S}\Big)+\Big(C-C_e+C_e \ln\frac{C_e}{C}\Big)+\Big(I_P-I_{Pe}+I_{Pe} \ln\frac{I_{Pe}}{I_P}\Big)+\Big(R-R_e+R_e \ln\frac{R_e}{R}\Big).\label{sys4aj}
\end{equation}
Note that $L(S_e,C_e,I_{Pe},R_e)=0$ and $W>0$ otherwise.
Differentiating equation (\ref{sys4aj}) with respect to $t$, yields
\begin{equation}
\frac{dL}{dt}=\Big(1-\frac{S_e}{S}\Big) \frac{dS}{dt}+\Big(1-\frac{C_e}{C}\Big) \frac{C}{dt}+ \Big(1-\frac{I_{Pe}}{I_P}\Big) \frac{dI_{Pe}}{dt}+\Big(1-\frac{R_e}{R}\Big) \frac{dR_{e}}{dt}.\label{sys4ak}
\end{equation}
Substituting expressions for $\frac{dS}{dt},~\frac{dC}{dt},~\frac{dI_P}{dt},~\text{and}~\frac{dR}{dt}$ in equation (\ref{sys4ak}), gives
\begin{align}
\frac{dL}{dt}=&\Big(1-\frac{S_e}{S}\Big)(\Lambda+\gamma R-(\alpha_P+\mu)S )  +\Big(1-\frac{C_e}{C}\Big) (\xi\alpha_PS-k_1C)\nonumber\\
&+ \Big(1-\frac{I_{Pe}}{I_P}\Big)((1-\xi)\alpha_PS+ \beta C-k_2I_P)+\Big(1-\frac{R_e}{R}\Big)(\pi C+\epsilon I_P-(\mu+\beta)R),\nonumber\\=&\Lambda+\gamma R+(\alpha_p+\mu)S_e+\alpha_P S+k_1C_e+(\beta+\pi)C+\epsilon I_P+(\mu+\beta)R_e-\Big(S\big(\alpha_P+\mu+\xi\alpha_P\frac{C_e}{C}\nonumber\\ &+(1-\xi)\alpha_P\frac{I_{Pe}}{I_P}\big)+\frac{S_e}{S}(\lambda+\gamma R)+k_1C+k_2I_P+\beta C\frac{I_{Pe}}{I_P}+(\pi C+\epsilon I_P)\frac{R_e}{R}+(\mu+\beta)R\Big),\nonumber\\
\frac{dL}{dt}=&A-B,\label{sys4al}
\end{align}
where 
\begin{align}&A=\Lambda+\gamma R+(\alpha_p+\mu)S_e+\alpha_P S+k_1C_e+(\beta+\pi)C+\epsilon I_P+(\mu+\beta)R_e, \nonumber\\
&B=S\big(\alpha_P+\mu+\xi\alpha_P\frac{C_e}{C}+(1-\xi)\alpha_P\frac{I_{Pe}}{I_p}\big)+\frac{S_e}{S}(\lambda+\gamma R)+k_1C+k_2I_P+\beta C\frac{I_{Pe}}{I_P}+(\pi C+\epsilon I_P)\frac{R_e}{R}+(\mu+\beta)R.\nonumber
\end{align}
Thus if $A<B$, we obtain  $\frac{dL}{dt}\leq 0$ and $\frac{dL}{dt}=0$ if and only if $ S=S_e, C=C_e,I_P=I_{Pe}, R=R_e.$ 
Therefore, the largest compact invariant set in $\{ (S, C, I_P, R)\in \Omega_2:\frac{dL}{dt}=0 \} $ is the singleton $\{{E}_{0P}\}, $ where ${E}_{0P}$  is the endemic equilibrium point of the system (\ref{sys4a}) and by LaSalle's invariance principle \cite{b11,k,b17}, it implies that ${E}_{0P}$ is globally asymptotically stable in $\Omega_2$ if $A<B$.\\
\end{proof}
\subsection{Analysis of the HIV/AIDS and PCP Co-Infection Model}
In this subsection, the disease-free equilibrium  $E_0$, the basic reproduction number $\mathcal{R}_0$, and the stability of model equilibria  (\ref{sys2}) are determined.
\subsubsection{Disease-Free Equilibrium  of HIV/AIDS and PCP Co-Infection Model}
At disease-free equilibria, that is in absence of both HIV and PCP; $C(t)=I_P(t)=I_H(t)=I_T(t)=I_{HP}(t)=I_{AP}(t)=T(t)=0$.
Equating the right hand side of system (\ref{sys2}) to zero and solving, gives $S(t)=\Lambda/\mu$ and $R(t)=0$. Therefore, the disease-free equilibrium is obtained as $E_0=\Big(\Lambda/{\mu},0,0,0,0,0,0,0,0,0\Big).$
\subsubsection{The Basic Reproduction Number of HIV/AIDS and PCP Co-Infection Model}
In this case, $\mathcal{R}_0$ defines the number of secondary HIV or Pneumonia co-infections due to a single HIV or PCP infective. Thus,  using the procedure earlier described in sub-subsection \ref{subsec1.22} on HIV/AIDS and PCP co-infection model (\ref{sys2}), we let $\mathcal{F}$ be the matrix for the rate of appearance of new PCP and HIV infections and  $\mathcal{V}$ be the matrix for the rate of other transfer terms such that; \\
 $\mathcal{F}=\begin{bmatrix}\xi\alpha_P(t) S(t)\\(1-\xi)\alpha_P(t) S(t)\\\alpha_H(t)\big(S(t)+R_P(t)\big)\\0\\\alpha_H(t)\big(C(t)+I_P(t)\big)+\alpha_P(t)I_H(t)\\\alpha_P(t)I_A(t)\end{bmatrix}$
and 
$\mathcal{V}=\begin{bmatrix}\big(\alpha_H(t)+\beta+\pi+\mu\big)C(t)\\ -\beta C+\big(\alpha_H+\epsilon+\mu+\nu_P\big)I_P(t)\\\big(\alpha_P(t)+\tau_1+\lambda_1+\mu\big)I_H(t)\\-\lambda_1I_H+\big(\alpha_P+\tau_2+\mu+\nu_A\big)I_A(t)\\ \big(\lambda_2+\tau_3+\mu+\nu_P\big)I_{PH}(t)\\-\lambda_2I_{HP}+\big(\tau_4+\mu+\nu\big)I_{AP}(t) \end{bmatrix}.$\\
The Jacobian matrix  $F$ of new infections at disease-free equilibrium is given by
\begin{align}
&F=\begin{bmatrix}\xi\rho c \omega&\xi\rho  c&0&0&\xi\rho c\theta_1&\xi\rho c\theta_2\\(1-\xi)\rho c\omega&(1-\xi)\rho c&0&0&(1-\xi)\rho c\theta_1&(1-\xi)\rho c\theta_2\\0&0&\eta k&a_2\eta k&a_1\eta k&a_3\eta_k\\0&0&0&0&0&0\\0&0&0&0&0&0\\0&0&0&0&0&0 \end{bmatrix}.\nonumber
\end{align}
The Jacobian matrix $V$ for the rate of transfer from one component to another at disease-free equilibrium is given by
\begin{align}
&V=\begin{bmatrix} k_1&0&0&0&0&0\\-\beta&k_2&0&0&0&0\\0&0&q_1&0&0&0\\0&0&-\lambda_1&q_2&0&0\\0&0&0&0&d_1&0\\0&0&0&0&-\lambda_2&d_2\\\end{bmatrix},\nonumber
\end{align}
where
$k_1=(\beta+\pi+\mu),k_2=(\epsilon+\mu+\nu_P),q_1=(\tau_1+\lambda_1+\mu),q_2=(\tau_2+\mu+\nu_A),d_1=(\lambda_2+\tau_3+\mu+\nu_P)$ and $d_2=(\tau_4+\mu+\nu)$.
Thus,  the basic reproduction number is given by
$$\mathcal{R}_0=\mathcal{K}(FV^{-1})=\max\{\mathcal{R}_{0P},\mathcal{R}_{0H}\}=\max\Big\{\frac{\rho c}{k_1k_2}\Big(\xi(\omega k_2+\beta)+(1-\xi)k_1\Big),\frac{\eta k}{q_1q_2}\Big(q_2+a_2\lambda_1\Big)\Big\},$$ where $\mathcal{K}$ is the spectral radius of $FV^{-1}$.
\subsubsection{Local stability of  HIV/AIDS and PCP Free Equilibrium}
\begin{Theorem} The disease-free equilibrium point $E_{0}$ of system (\ref{sys2}) is locally asymptotically stable whenever $R_{0} < 1$,  and otherwise unstable.\end{Theorem}
\begin{proof}
The Jacobian matrix for system (\ref{sys2}) evaluated at disease-free equilibrium is given by
\begin{equation}J_{(E_{0})}=\begin{bmatrix}-\mu & -\rho c\omega &- \rho c& \gamma&-\eta k&-\eta ka_2&0 &-(\rho c\theta_1+\eta k a_1)&-(\rho c\theta_2+\eta k a_3)&0\\0&\xi \rho c\omega -k_1&\xi\rho c&0&0&0&0&\xi\rho c\theta_1&\xi\rho c\theta_2&0\\0&(1-\xi) \rho c\omega -k_1&(1-\xi)\rho c&0&0&0&0&(1-\xi)\rho c\theta_1&(1-\xi)\rho c\theta_2&0\\0&\pi&\epsilon&-k_3&0&0&0&0&0&0\\0&0&0&0&\eta k-q_1&\eta ka_2&0&\eta ka_1&\eta ka_3&0\\0&0&0&0&\lambda_1&-q_2&0&0&0&0\\0&0&0&0&\tau_1&\tau_2&-\mu&0&0&0\\0&0&0&0&0&0&0&-q_3&0&0\\0&0&0&0&0&0&0&\lambda_2&-q_4&0\\0&0&0&0&0&0&0&\tau_3&\tau_4&-\mu\end {bmatrix},\label{sys4an}\end{equation}
where $k_1=(\beta+\pi+\mu),k_2=(\epsilon+\mu+\nu_P),k_3=(\gamma+\mu), q_1=(\tau_1+\lambda_1+\mu),q_2=(\tau_2+\mu+\nu_A),q_3=(\lambda_2+\tau_3+\mu+\nu_P)$ and $q_4=(\tau_4+\mu+\nu)$.\\
Thus, from the Jacobian matrix (\ref{sys4an}), we get the following characteristic polynomial
\begin{align}
P(\lambda)&=(\lambda+\mu)^3(\lambda+k_3)(\lambda+q_3)(\lambda+q_4)(\lambda^2+a_1\lambda+a_0)(\lambda^2+b_1\lambda+b_0),\label{sys4ao}
\end{align}
where $a_1=k_1+k_2-\rho c\big((1-\xi)+\xi\omega\big)>0,~\text{if}~k_1+k_2>\rho c\big((1-\xi)+\xi\omega\big), a_0=k_1k_2-\rho c\big((1-\xi)k_1+\xi(k_2\omega+\beta)\big)=k_1k_2(1-\mathcal{R}_{0P})>0,~\text{if}~\mathcal{R}_{0P}<1, b_1=q_1+q_2-\eta k>0,~\text{if}~q_1+q_2>\eta k,\label{sys4ap}$ and $b_0=q_1q_2-\eta k(q_2+a_2\lambda_1)=q_1q_2(1-\mathcal{R}_{0P})>0, ~\text{if}~\mathcal{R}_{0H}<1.$
This implies that $\lambda_1=-\mu<0$ (three times), $\lambda_2=-k_3<0,~\lambda_3=-q_3=<0,~\lambda_4=-q_4<0$ and by Routh-Hurwitz stability criteria, the eigenvalues of characteristic equations $\lambda^2+a_1\lambda+a_0=0$ and $\lambda^2+b_1\lambda+b_0=0$ are negative if $a_0>0,a_1>0, b_0>0$ and $b_1>0$ which is true when $\mathcal{R}_{0P}<1$ and $\mathcal{R}_{0H}<1$.
Therefore, since all eigenvalues of characteristic polynomial (\ref{sys4ao}) are negative for $\mathcal{R}_{0}<1$, the disease-free equilibrium point of system (\ref{sys2}) is locally asymptotically stable.

\end{proof}

\subsubsection{Existence of Endemic Equilibrium Point of System (\ref{sys2})}
The HIV/AIDS and PCP co-infection endemic equilibrium point $E_e=(S^*,C^*,I_P^*,R^*,I_H^*,I_A^*,I_T^*,I_{HP}^*,I_{AP}^*,T^*)$ materializes when PCP, HIV/AIDS and their co-infection persevere in the community. It has already been established from the analysis of sub-models (\ref{sys3a}) and (\ref{sys4a}) that the endemic steady states do not exist when ${\mathcal{R}_{0H}}<1$ and ${\mathcal{R}_{0P}}<1$ respectively. This in turn signifies that there is no endemic equilibrium point of the full co-infection model (\ref{sys2}) if ${\mathcal{R}_{0}}=\max\{{\mathcal{R}_{0H}},{\mathcal{R}_{0P}}\}<1.$\\
The analytical computation of the endemic equilibrium of the full model (\ref{sys2a}) in terms of  model parameters is strenuous; however, it exists when ${\mathcal{R}_{0H}}>1$ and ${\mathcal{R}_{0P}}>1$, that is  ${\mathcal{R}_{0}}=\max\{{\mathcal{R}_{0H}},{\mathcal{R}_{0P}}\}>1.$ The existence and stability of the endemic equilibrium point $E_e=(S^*,C^*,I_P^*,R^*,I_H^*,I_A^*,I_T^*,I_{HP}^*,I_{AP}^*,T^*)$ will be numerically scrutinized under numerical simulations.
\subsection{Sensitivity analysis}
In order to determine how to reduce the burden due to HIV/AIDS, PCP and  their co-infection,  we calculate the sensitivity indices of the basic reproduction number, $ \mathcal{R}_{0}$  with respect to the parameters in the model using the approach in \cite{b9,b23}. Sensitivity analysis determines parameters that have a high impact on $ \mathcal{R}_{0}$ and hence which parameters should be targeted for intervention strategies. The sensitivity index of $\mathcal{R}_0$ with respect to a parameter, $x$ is given by 
$\Lambda^{\mathcal{R}_0}_{x}=\frac{\partial \mathcal{R}_0}{\partial x}*\frac{x}{\mathcal{R}_0}.$

The sensitivity indices of $\mathcal{R}_{0H}$ and $\mathcal{R}_{0P}$ with respect to parameters are given in Table \ref{t1}.

\begin{table}[h!]
	\caption{Numerical values of sensitivity indices of $\mathcal{R}_{0H}$ and $\mathcal{R}_{0P}$}{\color{white} hhhhh}\label{t1}
\begin{center}   	
	\begin{tabular}{c*{4}{c}l}
		\hline
		 Parameter symbol&Sensitivity Index ($\mathcal{R}_{0H}$)&\vline&Parameter symbol&Sensitivity Index ($\mathcal{R}_{0P}$)&\\
		\hline
		$\eta$&$+1.0000$&\vline&$\rho$&+1.0000	&\\
		$k$&+1.0000& \vline&$c$&+1.0000	&\\
		$a_2$&+0.1519&\vline&$\xi$&+0.1953&\\
		$\lambda_1$&$-0.0747$&\vline&$\omega$&$+0.4360$&\\
                     $\tau_1$&$-0.5666$&\vline&$\nu_P$&$-0.1512$&\\
                     $\tau_2$&$-0.0368$&\vline&$\beta$&$-0.0224$&\\
                     $\lambda_1$&$-0.0747$&\vline&$\epsilon$&$-0.3024$&\\
                     $\nu_A$&$-0.0944$&\vline&$\pi$&$-0.0563$\\
		\hline	
	\end{tabular}
\end{center}	
\end{table}
It is noted that the value of $\mathcal{R}_0$ increases when parameter values $\eta$, $k$, $a_2$, $\rho$, $\xi$ , $c$ and $\omega$ increase while the other parameters values are kept constant since they have positive indices, and this implies that the endemicity of the disease is increased.\\ 
On the other hand, when parameters  $\tau_1$, $\tau_2$, $\lambda_1$, $\nu_A$, $\epsilon$, $\beta$, $\nu_P$ and $\pi$ are increased while the  other parameters values are kept constant, the value of $\mathcal{R}_0$ decreases implying that the endemicity of the disease is decreased.
\section{Numerical Simulations}\label{sec4}
In this section, we numerically simulate the stability of the endemic equilibrium point for the full co-infection model (\ref{sys2}), the contribution of PCP carriers to the HIV/AIDS-PCP co-infection burden and the effect of treatment on the control of the co-infection burden. With data set in Table \ref{t3} and initial conditions; $S_0=10000, C_0=200, I_{P0}=250, R_0=150, I_{H0}=400, I_{A0}=250, I_{T0}=300, I_{HP0}=350, I_{AP0}=150, T_0=150$, the model was simulated using ODE solvers coded in MATLAB numerical solver. 
\begin{table}[h!]
	\caption{Parameter values used in numerical simulations}{\color{white} hhhhh} \label{t3}
	\begin{center}	   	
		\begin{tabular}{c*{4}{c}l}
			\hline
			Parameter&Value(year$^{-1}$)& Source&Parameter&Value(year$^{-1}$)&Source\\
			\hline
			$\Lambda$& 2000&Assumed&$\gamma$&0.0621$$&Assumed\\
			$\mu$& 0.073& \cite{b2}&	$\eta$&0.075&\cite{b6}\\
			$k$& 1-5&Estimated&$\rho$&0.89-0.99&\cite{b5}\\
			$c$&1-50&Assumed&$\epsilon$&0.2&Assumed\\
			$\omega$& 0.41026&Assumed&$\pi$& 0.0115&Assumed\\
			$\beta$& 0.01096&Assumed&$\nu_P$&0.1&Assumed\\
			$\nu_A$&0.333&\cite{b1}&	$\nu$&0.42&\cite{b1}\\
			$\tau_1$&0.2&Estimated&$\tau_2$& 0.13&\cite{b3}\\
			$\tau_3$& 0.314&Assumed&$\tau_4$& 0.230&\cite{b3}\\	
			$\xi$& 0.338&Assumed&$\lambda_1,\lambda_2$&0.08,0.3105 &\cite{b1}\\
			$a_1,a_2,a_3$&1,1.2,1.4&Assumed&$\theta_1,\theta_2$&1,1.02&Assumed\\
		\hline
		\end{tabular}	\end{center}
\end{table}\\
\subsection{Stability of the Endemic Equilibrium Point}
Figure \ref{f2} shows the local stability of the endemic equilibrium of the HIV/AIDS and PCP co-infection model at $\mathcal{R}_{0H}=1.0021>1$ and $\mathcal{R}_{0P}=11.6601>1$. Since the system moves to equilibrium after 40 years from the initial populations, we conclude that the endemic equilibrium point is locally asymptotically stable.
\begin{figure}
\centering
	\includegraphics[scale=0.5]{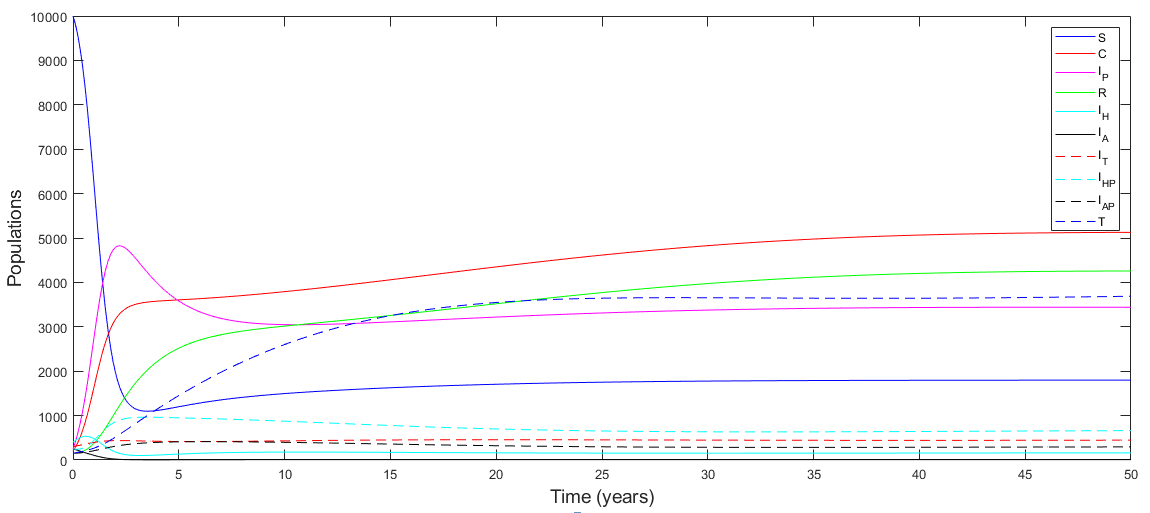}
\caption{The local stability of the endemic equilibrium point at $\mathcal{R}_{0H}=1.0021$ and $\mathcal{R}_{0P}=11.6601$.\label{f2}}
\end{figure}
\begin{figure}
\centering
	\includegraphics[scale=0.5]{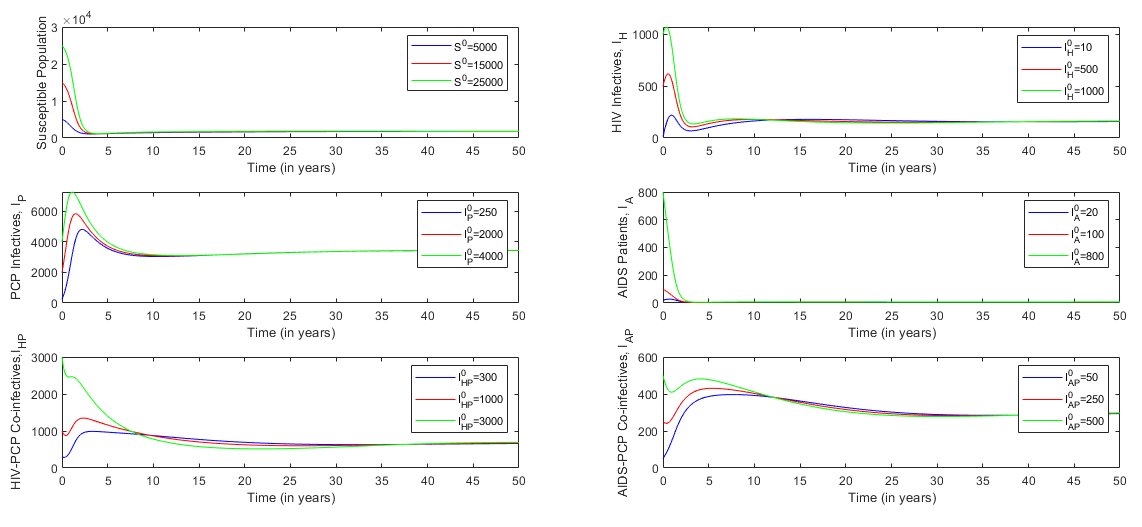}
\caption{The Global stability of the endemic equilibrium point.\label{f3}}
\end{figure}
Figure \ref{f3} indicates that the PCP and HIV/AIDS co-infection endemic equilibrium is globally stable since the system comes to  equilibrium  from any possible initial conditions.
\subsection{Visualization of the Role Played by PCP Carriers in Co-infection Dynamics of HIV/AIDS and PCP.}
\begin{figure}
\centering
	\includegraphics[scale=0.5]{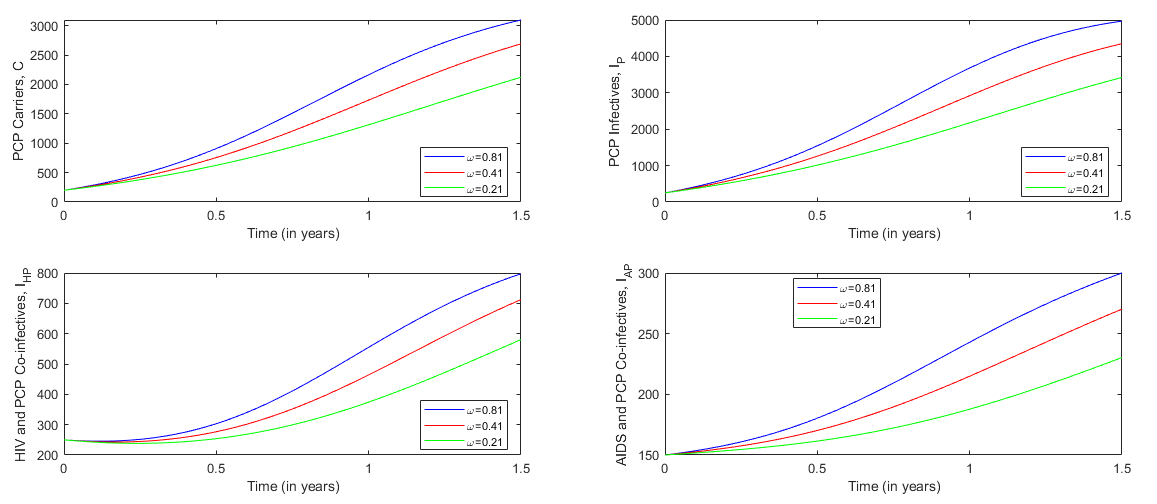}
\caption{Effect of PCP Carriers infectivity on Co-dynamics of HIV/AIDS and PCP.\label{f5}}
\end{figure}
Numerical simulations in Figure \ref{f5} show the effect of varying the infectious coefficient of PCP carriers on the PCP carrier, PCP infective and HIV/AIDS-PCP co-infected populations. As the value of $\omega$ is altered from $0.2-0.8$, Figure \ref{f5} affirms that the PCP carrier, PCP infective and the co-infected populations  also increase. This is a clear indicator that the contribution of PCP carriers in the transmission dynamics of PCP and its association with HIV/AIDS should not be given a blind eye.\\
\begin{figure}
\subsection{Visualization of the Effect of Treatment at Various Stages of Infections}
\centering
	\includegraphics[scale=0.5]{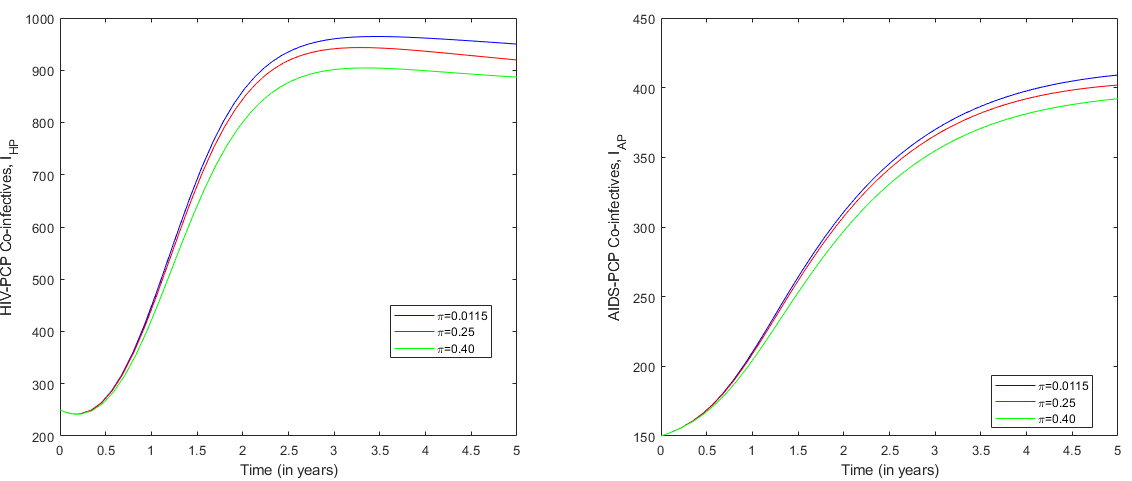}
\caption{Role of treating PCP carriers on HIV/AIDS-Pneumonia Co-dynamics.\label{f6}}
\end{figure}
\begin{figure}
\centering
	\includegraphics[scale=0.5]{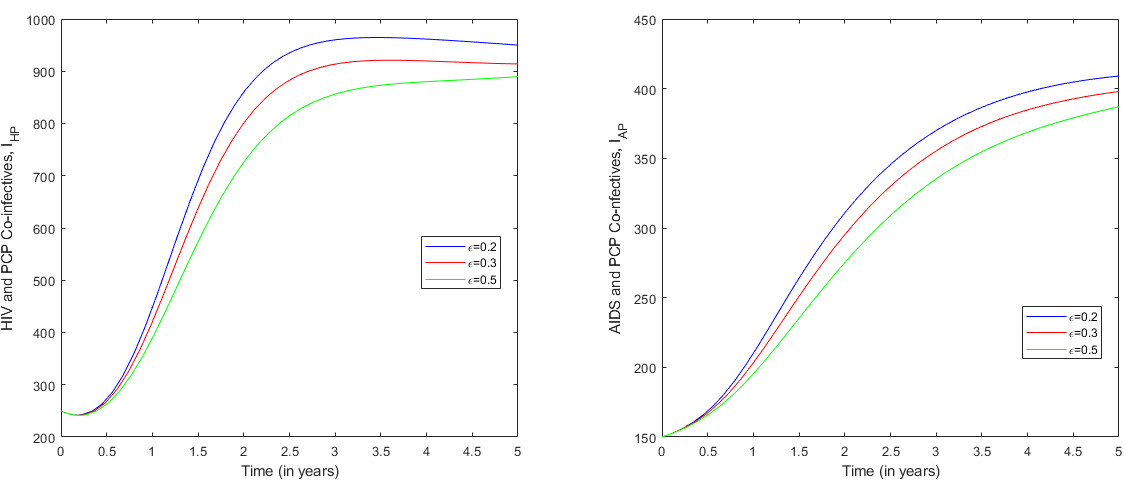}
\caption{Effect of treating PCP infectives on HIV/AIDS-Pneumonia Co-dynamics.\label{f7}}
\end{figure}
\begin{figure}
\centering
	\includegraphics[scale=0.5]{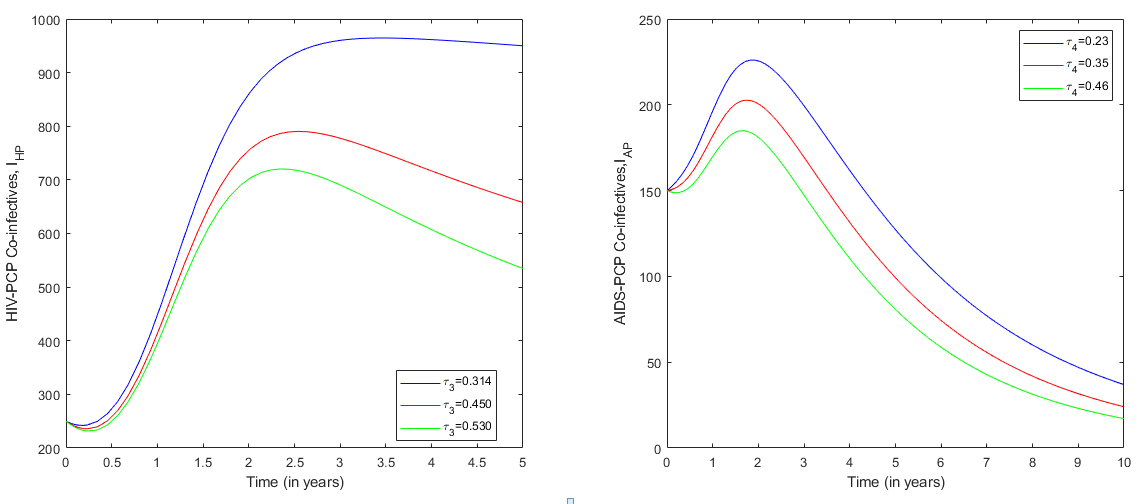}
\caption{Effect of Dual treatment on HIV-PCP Co-dynamics.\label{f8}}
\end{figure}
Simulations in Figures \ref{f6} to \ref{f8} show the effect of interrupting PCP through treatment effort at different phases of infection. Increasing the treatment rate of PCP carriers $\pi$ from $0.0115- 0.40$, Figure \ref{f6} shows that the number of dually infected individuals reduce. Figure (\ref{f7}) shows that increasing the treatment rate of PCP infectives $\epsilon$ from $0.2- 0.5$, results in a decline in number HIV/AIDS-PCP co-infected individuals.\\
The impact of increasing dual treatment rates $\tau_3$ from $0.314- 0.530$ and $\tau_4$ from $0.23 - 0.46$  are observed to reduce the number of HIV-PCP and AIDS-PCP co-infectives in Figure \ref{f8}.

\section{Discussion and Conclusion}\label{sec5}
In this paper, an HIV/AIDS and PCP co-infection model is derived and analyzed.  The basic properties of the model are shown and it is established that the model is biologically meaningful and well posed.\\ The disease-free steady states $E_{0H}$ for HIV/AIDS sub-model, $E_{0P}$ for PCP sub-model, $E_{0}$ for the dual infection and their corresponding reproduction numbers $\mathcal{R}_{0H}$,  $\mathcal{R}_{0P}$, $\mathcal{R}_{0}$ are derived. It is established that the disease-free steady states are locally asymptotically stable when their respective basic reproduction numbers are less than a unity. The equilibria $E_{0H}$ and $E_{0P}$ are globally stable whenever $\mathcal{R}_{0H}\leq 1$ and $\mathcal{R}_{0P}\leq 1$ respectively. This means that HIV/AIDS infection,  PCP infection and the HIV/AIDS-PCP co-infection will not prevail and will eventually be wiped out in the community.\\ Further analysis shows that the disease persistent equilibrium points $E_{eH}$ for HIV/AIDS sub-model and $E_{eP}$ for PCP sub-model, are unique if their corresponding reproduction numbers $\mathcal{R}_{0H}>1$ and $\mathcal{R}_{0P}>1$.
By use of suitable Lyapunov functions, the equilibria  $E_{eH}$ and  $E_{eP}$ are globally asymptotically stable whenever $\mathcal{R}_{0H}>1$ and $\mathcal{R}_{0P}> 1$ respectively. The local and global stability of the endemic equilibrium point $E_e$ of the co-infection model is numerically determined and is shown to be stable if $\mathcal{R}_0=\max\{\mathcal{R}_{0H},\mathcal{R}_{0P}\}>1$  in figures \ref{f2} and  \ref{f3}. This suggests that the HIV/AIDS infection, PCP infection and HIV/AIDS-PCP co-infection will surge when more than one HIV/AIDS or PCP or co-infected individuals are introduced in the community.\\
The sensitivity analysis of the basic reproduction number was carried out and results indicated that treatment is important in the reduction of PCP across all affected sub-groups. The co-efficient of transmission of PCP carriers gave a positive sensitivity index which implied that PCP carriers are silent spreaders of the infection and need to be interrupted in order to scale down the PCP incidence rate.\\
Numerical simulations show that increased infectivity of PCP carriers increases the number of PCP carriers, PCP infectives, HIV-PCP co-infectives and AIDS-PCP co-infectives as shown in figure \ref{f5}. This calls for an alarm to the policy makers  to introduce case finding strategy for PCP carriers since their contribution to PCP infection and its association with HIV/AIDS is enormous. This can also be achieved by testing all HIV infected individuals for a possible PCP infection and all positive cases subjected to PCP treatment.
Furthermore, simulations dictate that increased rates of treatment  of PCP carriers, PCP infectives, and dually infected individuals is  of  a great deal towards reducing the burden of PCP infection and its close association in HIV/AIDS patients.\\
We recommend that if HIV/AIDS and PCP co-infection burden is to be managed, all HIV infected individuals be subjected to obligatory PCP diagnosis and positive cases treated accordingly. It is clear that interrupting PCP at all contagious phases by treating PCP carriers, PCP infectives and HIV/AIDS-PCP co-infectives will scale down the burden of PCP on individuals living with HIV/AIDS.

The model is not without limitations. The model did not take  into consideration of immigration of carriers, infected and co-infected individuals into the system, HIV infected individuals were not split into chronic and acute subgroups, individuals who default  ART were also not included. Incorporating these processes will undoubtedly facilitate in the understanding of HIV/AIDS and PCP co-infection  transmission and control dynamics.
\section*{Author Contributions}
\textbf{Michael Byamukama:} Conceptualization, analysis and manuscript preparation.\\
\textbf{Damian Kajunguri:} Supervision and interpretation of results.\\
\textbf{ Martin Karuhanga:} Supervision, analysis and proofreading.
\section*{Conflict of Interest}
The authors declare that they have no conflicting interests.

\end{document}